\documentclass[12pt]{amsart}
\usepackage{amscd}
\usepackage{amssymb}
\usepackage{amsmath}

\newtheorem{theorem}{Theorem}[section]
\newtheorem{corollary}[theorem]{Corollary}
\newtheorem{lemma}[theorem]{Lemma}
\newtheorem{proposition}[theorem]{Proposition}

\newtheorem{remark}[theorem]{Remark}
\def \nmatrix #1 #2 #3 #4 {
#1=\begin{pmatrix}
 0 & #3 & #4 \\
 0 & 0  & #2 \\
 0 & 0  & 0
\end{pmatrix}
}
\def \End {\operatorname{End}}
\def \a{\operatorname{\alpha}}
\def \g{\operatorname{\gamma}}
\def \SL {\operatorname{SL}}
\def \SU {\operatorname{SU}}
\def \GL {\operatorname{GL}}
\def \Ad{\operatorname{Ad}}
\def\Tr{\operatorname{Tr}}

\def \nc {A_{\theta}}
\begin{document}

\title [Non-commutative sigma model]
       {A note on non-linear $\sigma$-models in noncommutative geometry}

\begin{abstract}
 We study non-linear $\sigma$-models defined on non-commutative torus as two dimensional string world-sheet.  We consider a quantum group as a noncommutative space-time as well as two points, a circle, and a (irrational) noncommutative trous. Using the established results  we show that the trivial harmonic unitaries of the noncommutative chiral model, which are known as local minima, are not global minima by comparing those with the symmetric unitaries coming from instanton solutions of non-commutative Ising model, which corresponds to the two points target space.  In addition, we introduce a $\mathbb{Z}^2$-action on field maps to $\nc$ and show how it acts on solutions of various Euler-Lagrange equations. 
\end{abstract}

\author { Hyun Ho \quad Lee }
\address {Department of Mathematics\\
         University of Ulsan\\
         Ulsan, Korea 608-749 }
\email{hadamard@ulsan.ac.kr} 
\date{Sep15, 2014}
\maketitle
\maketitle

\section {Introduction}

Following  Connes' foundational work of general formulation of action functionals in noncommutative space-time \cite{Co}, classical non-linear $\sigma$-models in (open) string theory can be formulated in the framework of noncommutative geometry. The noncommutative counterpart of $\sigma$-model is called noncommutative $\sigma$-model and consider as field maps $*$-homomorphisms from a noncommutative space-time to a noncommutative string world-sheet. It is worth to mention Mathai and Rosenberg's recent works on T-duality demonstrating an increasing role of noncommutative space-time in quantum field theory\cite{MaRo1, MaRo2}.\\ 

Concrete examples of noncommutative $\sigma$-model appeared in \cite{DKL}, and more sophisticated works were done by Mathai and Rosenberg \cite{MaRo}. They also give a general set-up for noncommutative $\sigma$-model which is our starting point of the theory. However, we will not consider the Wess-Zumino term here so that the physical model is not treated. Rather we are strongly motivated by \cite{Ro} and keep our intersts on noncommutative harmonic maps. \\

As many of classical elliptic PDEs arise from variational problem in Riemmanian geometry, and the field equations of physical theories, different non-commutative sigma models give rise to various noncommutative nonlinear elliptic PDEs which are the Euler-Lagrange equations for stationary points of the enegy functionals. In this paper we consider the sigma-models based on the noncommutative torus $\nc$, which is the univeral $C\sp*$-algebra of two unitaries $U$ and $V$ such that $UV=e^{2\pi i \theta}VU$ for $\theta \in \mathbb{R} \setminus \mathbb{Q}$ or its smooth dense subalgebra $\nc^{\infty}$(for more rudiments refer \cite{Co1, Rie2, Rie3}). In other words, we fix a source space or a world-sheet as the noncommutative torus for an irrational $\theta$ and for the various target spaces we observe the Euler-Lagrange equations. In particular, we try to investigate the case of Woronowicz's compact quantum matrix group $\SU_q(2,\mathbb{C})$. By a careful treatment of instantons of noncommutative Ising model we show that trivial harmonic unitaries are not global minima while they are strict local minima \cite{Ro}\cite{Li}. Related to a major issue of solving these equations I introduce a $\mathbb{Z}^2$-action on field maps to a noncommutative torus. At least I suspect that this method  could be a general prescription for generating a family of solutions from a known solution, thereby in the remaining part we demonstrates that it is true for all models known so far.

\section{Non-commutative $\sigma$-models on the non-commutative torus}
\subsection{An unified approach to non-commutative sigma model}
In this section, following \cite{MaRo} we reivew two dimensional non-linear $\sigma$-models in the framework of non-commutative geometry with emphasis on models defined on the non-commutative torus.
Recall that the spectral triple $(\mathcal{A}, \mathcal{H}, D)$ consists of an involutive algebra $\mathcal{A}$ represented as bounded operators on a Hilbert space $\mathcal{H}$ and a self-adjoint operator $D$ with compact resolvent such that the commutators $[D, a]$ are bounded for all $a\in \mathcal{A}$. A spectral triple  $(\mathcal{A}, \mathcal{H}, D)$ is said to be  \emph{even} if the Hilbert space $\mathcal{H}$ is endowed with a super-grading $\gamma$ which commutes with all $a\in \mathcal{A}$ and anti-commutes with $D$. In addition, we assume that $(\mathcal{A}, \mathcal{H}, D)$ is $(2,\infty)$-summable, which means roughly $\Tr_{\omega}(a|D|^{-2})<\infty$, where $\Tr_{\omega}$ is the Dixmier trace. Let $\Tr$ be the Dixmier trace composed with $D^{-2}$. Note that 
\[\psi_2(a_0,a_1,a_2)=\Tr((1+\gamma)a_0[D,a_1][D,a_2])\] defines a positive Hochschild 2-cocycle on $\mathcal{A}$, where $\gamma=\begin{pmatrix} 1& 0 \\ 0 & -1\end{pmatrix}$ is the grading on $\mathcal{H}$. The positivity of $\psi_2$ means that $<a_0\otimes a_1,b_0 \otimes b_1 >=\psi_2(b_0^*a_0,a_1,b_1^*) $ defines a positive sesquelinear form on $\mathcal{A}\otimes \mathcal{A}$.  Then a (noncommutative) sigma-model consists of homomorphisms $\phi: \mathcal{B} \to \mathcal{A}$ with the target the given even $(2,\infty)$-spectral triple $(\mathcal{A}, \mathcal{H}, D)$, and a positive element $G \in \Omega^2(\mathcal{B})$ in the space of universal $2$-forms on $\mathcal{B}$.  The enegy functional in the $\sigma$-model is given by 
\[\mathcal{L}_{G, D}(\phi)=\phi^*(\psi_2)(G)\ge 0.\] 
Since $\mathcal{A}$ represents a ``string world-sheet'', it seems natural to consider the non-commutative torus as a source space. Then $\mathcal{B}$ corrsponds to a (non-commutative or commutative) space-time, and we give several models corrsponding to different space-times with the (fixed) source space $\mathcal{A}=\nc$.  The goal of non-commutative $\sigma$-model is to find the critical points of the functional $\mathcal{L}_{G, D}$ and we call such critical points harmonic maps \cite{DKL2, MaRo}. Naturally using variational method each model gives rise to the Euler-Lagrange equation(s). 

Let us choose and fix a spectral triple $(\nc^\infty, \mathcal{H}, D)$ given by  $\mathcal{H}=L^2(\nc, \tau)\otimes \mathbb{C}^2$ , which is the tensor product of the GNS representation space coming from the trace $\tau$ of $\nc$ and the $2$-dimensional representation space of spinors, and  $D=\sigma_1\delta_1+\sigma_2\delta_2$, where 
\[\sigma_1=\begin{pmatrix} 0 & 1 \\ 
1 & 0 \end{pmatrix}, \quad \sigma_2=\begin{pmatrix} 0 & -i \\ 
i & 0 \end{pmatrix}.\] Here $\delta_1$ and $\delta_2$ are the infinitesimal generators coming from the dynamical system $(\mathbb{T}^2,  \nc)$,  which are defined on $\nc^{\infty}$ by the formulars
\[\delta_1(U)=2\pi i U, \quad \delta_1(V)=0,\quad  \delta_2(V)=2 \pi i V,\quad  \delta_2(U)=0.\]

1. Noncommutative Ising Model \\
The first example of noncommutative $\sigma$-model is the model with  the two points target space or $\mathbb{C}^2$ introduced by Dabrowski, Krajewski, Landi \cite{DKL}.  This is the simplest model we can think of, and any field map  $\phi$ from the algebra $\mathcal{B}=\mathbb{C}^2$ of functions over a two points space to $\mathcal{A}$ is determined by a projection in $\nc$ which is the image of  the characteristic function $e$ on one point under $\phi$.  Let us denote by $p$ the field map in this case. We choose a metric $G=dede \in \Omega_2(\mathcal{B})$. Then the action functional  is given by 
\[\begin{split}
 \mathcal{L}_{G,D}(\phi)&=S_D(p)=\Tr((1+\gamma)[D, p][D, p])\\
  &= \Tr(\delta_1(p)\delta_1(p)+\delta_2(p)\delta_2(p))\\
&=\tau(\delta_1(p)\delta_1(p)+\delta_2(p)\delta_2(p)).
\end{split}
\]
The Euler-Lagrange equation is 
\begin{equation}\label{E:projection}
p(\Delta p)-(\Delta p)p=0,
\end{equation} where $\Delta=\delta^2_1+\delta^2_2$ is the Laplacian. Note that it is a non-linear equation of second order and  difficult to solve directly.   For our purpose let us further review the instanton solution carrying a 'topological charge'\cite{DKL}. 

In general, given a projection $p$ in a connected componet of the space $P_{\theta}$ of all projections of $\nc$,  a topological charge or the first Chern number $c_1(p)$ is defined by \cite{Co1}
\[\frac{1}{2\pi i} \tau(p[\delta_1(p)\delta_2(p)-\delta_2(p)\delta_1(p)]). \]
It is easy to obtain the inequality 
\begin{equation} \label{E:enery>chern}
S_D(p)\ge - 2\pi c_1(p).
\end{equation} The equality in  (\ref{E:enery>chern}) happens when the so-called self duality equation
$\displaystyle \overline{\partial}(p) p=0$ is satisfied where $\overline{\partial}=1/2 (\delta_1+i \delta_2)$. 
A way of solving the self-duality equation was suggested in \cite{DKL}, here we want to gve a more systematic treatment with Theorem \ref{T: estimate} in mind . 

Let $\alpha$ be any non-zero real number, and $p$ and $q$ be a pair of integers which generates $\mathbb{Z}$, such that $q\alpha+p \neq 0$ and $q\neq 0$. Fix $\displaystyle \eta=\frac{1}{q\alpha+p}$. If we suppose $\beta=(a\alpha+b)\eta$ for any integers $a, b$ such that $ap-bq=\pm 1$, then  $A=M_d(A_{\beta})$ and $B=A_{\alpha}$ are strongly Morita equivalent via the module the completion of $C_c(\mathbb{R}\times\mathbb{Z}_{dq})$ under either $\underset{A}{<}, \, >$ or $<, \,\underset {B}{>}$(see \cite[Theorem 1.1]{Rie1}). The module is denoted by $V_{\alpha}(p,q; d)$ and is finitely generated and projective as a right $A_{\alpha}$-module, and that its full endomophism ring is $M_{d}(A_{\beta})$. We are particularly interested in  the case $d=1, q=1,p=0, a=0, b=-1$. Then the module $V(0,1;1)$,  which was used in \cite{Rie2}, is just the completion of $C_c(\mathbb{R})$ and if $\beta=\theta$, then $\alpha=-\frac{1}{\theta}$. Thus we obtain the Morita equivalence of $A=\nc$, and  $B=A_{-\frac{1}{\theta}}$.  Actually we can obtain the smooth version of this relation by replacing $A$ and $B$ by the smooth dense subalgebras $A^{\infty}$ and $B^{\infty}$ via  the module $ \Xi=\mathcal{S}(\mathbb{R})$. Since we only need the smooth version, we shall omit $\infty$ in the notation. Recall that the isomorphism between $A$ and $\End_B(\Xi)$ is implemented by viewing $a \in A$ as the left multiplication on $\Xi$.   Now if $\xi \in \Xi$ such that $<\xi, \xi \underset{B}{>}$ is invertible, then $x_1=\xi <\xi, \xi \underset{B}{>}^{-1},y_1=\xi$ satisfies $<x_1, y_1\underset{B}{>} =1$. Thus if we consider an endomorphism $\xi <\xi, \xi \underset{B}{>}^{-1}<\xi, \cdot \underset{B}{>} $, then it corrsponds to a projection  $p=\underset{A} {<} x_1, y_1>.$ Using this relation, we can show that $\displaystyle \overline{\partial}(p) p=0$ if and only if $\overline{\nabla}\xi=\xi \cdot b$ for some $b \in B$. Here $\overline{\nabla}=1/2(\nabla_1+\nabla_2)$ where $\nabla_i$'s define a connection on $\Xi$ \cite{CoRie, Rie3}. By taking $b=\lambda \in \mathbb{C}$, this amounts to solve the simple differential equation \[ \frac{d\xi}{dt}+(2\pi \theta t+2i \lambda)\xi=0.\]
The solutions $\xi_{\lambda}=Ce^{-\pi \theta t^2-2i\lambda t}$ are gaussians ($C$ is a constant), and for  $\theta$ small enough $<\xi_{\lambda}, \xi_{\lambda} >$ are invertible. Then we call the projections coming from such gaussians ``instanton''solutions. \\

2. Noncommutative  principal chiral model

Now let us consider the target space the circle $\mathbb{T}$ or $B=C(\mathbb{T})$ \cite{DKL, Ro}. Then the field map or homomorphism $\phi$ from $C(\mathbb{T})$ to $\nc$ is determined by the image of the generator $u$ viewing $C(\mathbb{T})=C^*(u)$.  Let $\phi(u)=W$. If we choose a positive element $du^*du \in \Omega^2(B)$, then the action functional is 
\[\begin{split}
 \mathcal{L}_{G, D}(\phi)&=2E(W)=\Tr((1+\gamma)[D,\phi(u^*)][D,\phi(u)])\\
&=\Tr((\delta_1(W))^*\delta_1(W)+(\delta_2(W))^*\delta_2(W))\\
&=\tau((\delta_1(W))^*\delta_1(W)+(\delta_2(W))^*\delta_2(W)),
\end{split}\]
and the Euler-Lagrange equation is 
\begin{equation}\label{E:harmonic}
\sum_{j=1}^2 \delta_j((\phi(u))^*\delta_j(\phi(u))=W^*(\Delta W)+(\delta_1(W))^*\delta_1(W)+(\delta_2(W))^*\delta_2(W)=0.
\end{equation}
A unitary  is  harmonic if it satisfies (\ref{E:harmonic}).  The scalar multiples of $U^mV^n$ are harmonic and is called trivial solutions. Though there are  solutions of other type (see Theorem \ref{T:symmetricharmonic} below), we would like to know if there is another way to construct a solution.\\

3. Noncommutative space-time; noncommutative torus 

Although the above two models are already non-trivial, the target spaces are commutative. The sigma-model with a noncommutative space-time appeared in \cite{MaRo} for the first time. In particular, Mathai and Rosenberg considered the noncommutative torus as the noncommutative space-time and they classified the homomorphisms between two different noncommutative tori. Here we only deal with the case that the field map is an endomorphsim on a fixed noncommutative torus $\nc$. In this case, the action functional is similar to the noncommutative chiral model since $B=\nc$ is generated by two unitaries $U, V$. By taking a positive element $dU^*dU+dV^*dV \in \Omega^2(B)$, the action functional $\mathcal{L}_{G, D}(\phi)$ for a unital $*$-endomorphism $\phi$ of $\nc$ is given by 
\[\begin{split}
\mathcal{L}_{G,D}&=\Tr((1+\gamma)[D,\phi(U^*)][D,\phi(U)])+ \Tr((1+\gamma)[D,\phi(
V^*)][D,\phi(V)])\\
&=\sum_j \Tr((\delta_i(\phi(U))^*\delta_i(\phi(U)))+\sum_j \Tr (\delta_j(\phi(V))^*\delta_j(\phi(V))).  \end{split}
\]
Then the Euler-Lagrnage equations for $\phi$ to be a harmonic map are 
\begin{equation}\label{E:harmonicmap}
0= \sum_{j=1}^2\{ \Tr(A \delta_j[\phi(U)^*\delta_j(\phi(V))]) +\Tr(B\delta_j[\phi(V)^*\delta_j(\phi(V))])\} 
\end{equation}
where $A, B$ are self-adjoint elements in $\nc$ constrained to satisfy the equation
\[ A-\phi(V)^*A\phi(V)=B-\phi(U)^*B\phi(U).\]

The interesting feature of the equation (\ref{E:harmonicmap}) is the appearance of the ``integral'' (compare with (\ref{E:projection}) and (\ref{E:harmonic}).)  Moreover, the constraint equation makes it more difficult to solve the Euler-Lagrange equations. It is known that the $*$-automorphism $\phi_A$ defined by $\phi_A(U)=U^pV^q$ and $\phi_A(V)=U^rV^s$, with $A=\begin{pmatrix} p & q \\ r & s \end{pmatrix} \in \SL(2, \mathbb{Z})$ is a harmonic map. 

4. A quantum group for the target space

Motivated by Mathai-Rosenberg model we try another noncommutative space-time so called Woronowicz's compact quantum group $\SU_{q}(2,\mathbb{C})$,  which is the universal $C\sp*$-algebra generated by $\a, \g$ subject to the condition that the following matrix is unitary;
\[
\begin{pmatrix}
\a & -q\g^* \\
\g & \a^*
\end{pmatrix}.
\]
Hence $\a$ and $\g$ satisfy
\begin{equation} \label{E:quantum}
 \begin{gathered}
\a^*\a+\g^*\g=1, \quad \a\a^*+q^2\g\g^*=1, \\
\g^*\g=\g\g^*,\quad  \a\g=q\g\a,\quad  \a\g^*=q\g^*\a.
\end{gathered}
\end{equation}
Usually, $q$ is a continuous parameter whose range is $[-1,1]$. Then we consider morphisms from $B=\SU_{q}(2,\mathbb{C})$ to $A=\nc$. Unfortunately it is observed that any morphism from $\SU_{q}(2,\mathbb{C})$ to $\nc$ boils down to the case $q=1$ because of the second and fourth conditions in (\ref{E:quantum}). Hence $\SU_1(2,\mathbb{C})$ is the commutative $C\sp*$-algebra $C(\SU(2,\mathbb{C}))$ which is generated by two coordinate functions $\a: \begin{pmatrix}
a & -\bar{b} \\
b & \bar{a}
\end{pmatrix} \in \SU(2,\mathbb{C}) \to a, \g: \begin{pmatrix}
a & -\bar{b} \\
b & \bar{a}
\end{pmatrix} \in \SU(2,\mathbb{C})\to b$. 

Although any morphism $\phi$ is determined by its images of $\a, \g$, we rather restrict ourselves to the smaller class of morphisms, so called the coercive maps which sends each of $\a$ and $\g$ to a scalar multiple of a unitary in $\nc$. Thus if $\phi(\a)=\mu u$, $\phi(\g)=\nu v$ for $\mu, \nu \in \mathbb{C}\setminus \{0\}$ and $u, v \in \nc$, then  we must have $|\mu|^2+|\nu|^2=1$. We note that the set of coersive maps is not empty (see Theorem \ref{T:coercivemap} below).
To define the energy functional, we take $G=d\a^*d\a+d\g^*d\g \in \Omega^2(B)$. Then we have

 \begin{equation}\label{E:SU(2)}
\begin{split}
 \mathcal{L}_{G,D}(\phi)&=\psi_2(1,\phi(\a^*), \phi(\a))+\psi_2(1,\phi(\g^*),\phi(\g)) \\
&=\sum_{j=1}^2 \Tr(\delta_j(\phi(\a^*))\delta_j(\phi(\a)))+\sum_{j=1}^2 \Tr(\delta_j(\phi(\g^*))\delta_j(\phi(\g))).
\end{split}
\end{equation} We determine the Euler-Lagrange equation for the energy functional of coercive maps $ \phi: \SU(2,\mathbb{C}) \to \nc$. We note that even though we take a commutative space-time one can observe the ``integral'' $\Tr$ in the Euler-Lagrange equation while Mathai and Rosenberg have thought that the appearance of $\Tr$ in the Euler-Lagrange equation might come from a non-commutative space-time.    
\begin{proposition}\label{P:SU(2)}
Let $\mathcal{L}_{G,D}$ denote the energy functional given by (\ref{E:SU(2)}) for coercive maps $\phi$ from $\SU(2,\mathbb{C})$ to $\nc$. Then the Euler-Lagrange equations for $\phi$ to be a harmonic map are 
\[0=\sum_{j=1}^2 \{ \Tr(A\delta_j[\phi(\a)^*\delta_j(\phi(\a))])\}+\sum_{j=1}^2 \{ \Tr(B\delta_j[\phi(\g)^*\delta_j(\phi(\g))])\}\]
where $A, B$ are self-adjoint elements in $\nc$ constrained to satisfy the equations
\begin{equation} \label{E:constraints}
 \begin{gathered}
   \phi(\a)A\phi(\g)-\phi(\g)\phi(\a)A=\phi(\g)B\phi(\a)-\phi(\a)\phi(\g)B, \\
 \phi(\a)A\phi(\g^*)-\phi(\g^*)\phi(\a)A=\phi(\a)B\phi(\g^*)-B\phi(\g^*)\phi(\a).
\end{gathered}
\end{equation}
\end{proposition}

\begin{proof}

For a fixed pair $(\mu,\nu)$, if we consider  an 1-parameter  family of $\phi$, then it is of the form $\phi_t(\a)=e^{i h(t)}\phi(\a)$ and $\phi_t(\g)=e^{ i  g(t)}\phi(\g)$  where h(t), g(t) are 1-parameter families of self-adjoint elements with $h(0)=g(0)=0$.  Using $\phi_t(\a)=\phi(\a)(1+ith'(0)+O(t^2)), \phi_t(\g)=\phi(\g)(1+itg'(0)+O(t^2))$, we can deduce the Euler-Lagrange equation using the same proof of \cite[Proposition 3.9]{MaRo}. The diffences come from the constraints given by (\ref{E:quantum}). It is easy to check the first condition $\phi_t(\a^*)\phi_t(\a)+\phi_t(\g^*)\phi_t(\g)=I$ is automatically satisfied. Similarly the second and third conditions holds. But   
\[ \begin{split}
0&=\frac{d}{dt}\Big|_{t=0}\phi_t(\a\g)-\phi_t(\g\a)\\
&=\lim_{t \to 0} \frac{(\phi(\a)+it\phi(\a)h'(0)+O(t^2))(\phi(\g)+it\phi(\g)g'(0)+O(t^2))-\phi(\a)\phi(\g)}{t}\\
&-\lim_{t \to 0}\frac{(\phi(\g)+it\phi(\g)g'(0)+O(t^2))(\phi(\a)+it\phi(\a)h'(0)+O(t^2))-\phi(\g)\phi(\a)}{t}\\
&=\lim_{t \to 0}\frac{it\phi(\a)h'(0)\phi(\g)+it\phi(\a)\phi(\g)g'(0)-it\phi(\g)g'(0)\phi(\a)-it\phi(\g)\phi(\a)h'(0)}{t}\\
&=i[\phi(\a)h'(0)\phi(\g)+\phi(\a)\phi(\g)g'(0)-\phi(\g)g'(0)\phi(\a)-\phi(\g)\phi(\a)h'(0)]
\end{split}
\]
Similarly, from $\displaystyle 0=\frac{d}{dt}\Big|_{t=0}\phi_t(\a\g^*)-\phi_t(\g^*\a)$ we can deduce that 
\[\phi(\a)h'(0)\phi(\g^*)-\phi(\g^*)\phi(\a)h'(0)+ g'(0)\phi(\g^*)\phi(\a)-\phi(\a)g'(0)\phi(\g^*)=0.\]
\end{proof}
As we promised, we show that there are infinitely many  coercive maps $\phi$ which are  solutions of the Euler-Lagrange equation.
\begin{theorem}\label{T:coercivemap}
Let $\phi_A$ be a map defined by $\phi_A(\a)=\frac{1}{\sqrt{2}}U^pV^q$ and $\phi_A(\g)=\frac{1}{\sqrt{2}}U^rV^s$, with $A=\begin{pmatrix} p & q \\ r & s \end{pmatrix}  \notin \GL(2,\mathbb{Z})$. Then $\phi_A$ is a critical point of $\mathcal{L}_{G,D}$ given by (\ref{E:SU(2)}).
\end{theorem}
\begin{proof}
 
\[ \begin{split}
\phi_A(\a)\phi_A(\g)&=\frac{1}{\sqrt{2}}U^pV^q \frac{1}{\sqrt{2}}U^rV^s=\frac{1}{2}e^{-2\pi i\theta
 qr}U^{p+r}V^{q+s}\\
\phi_A(\g)\phi_A(\a)&=\frac{1}{\sqrt{2}}U^rV^s \frac{1}{\sqrt{2}}U^pV^q=\frac{1}{2}e^{-2\pi i\theta
 ps}U^{p+r}V^{q+s}
\end{split}\]
Since $ps-rq=0$, $e^{2 \pi i \theta ps}=e^{2 \pi i \theta qr}$ and $\phi_A(\a)\phi_A(\g)=\phi_A(\g)\phi_A(\a)$. 
Similarly, we can check that $\phi_A(\a)\phi_A(\g^*)=\phi_A(\g^*)\phi_A(\a)$. Thus $\phi_A$ is a homomophism from $\SU(2,\mathbb{C})$ to $\nc$ and is a coercive map. In addition,
\[\begin{gathered}
\phi_A(\a^*)\delta_1(\phi_A(\a))= \pi i p \\
\delta_2(\phi_A(\a^*))\phi_A(\a)= \pi i q \\
\delta_1(\phi_A(\g^*))\phi_A(\g)= \pi i r \\
\delta_2(\phi_A(\g^*))\phi_A(\g)= \pi i p
\end{gathered}\] 
Applying any derivation $\delta_j$, $j=1,2$, to any of the terms above gives zero, since they are all constants. Thus for any $A,B$ satisfying the constraints (\ref{E:constraints}) the Euler-Lagrange equation is satisfied. 
\end{proof}
Mathai and Rosenberg raised a question whether there are other critical points aside from trivial solutions in \cite{MaRo} and the same question applies to our situation. Are there solutions other than $\phi_A$? 
 
\subsection{A construction of harmonic unitaries}
The following theorem is not new, it is mentioned in the paper \cite{DKL} shortly without proof. Here we give a full proof.   
\begin{theorem}\label{T:symmetricharmonic}
Let $W=1-2p$, where $p$ satisfies $(\Delta p)p-p(\Delta p)=0$. Then $W$ is a harmonic unitary.
\end{theorem}
\begin{proof}
Note that $\delta_i(W)=\delta_i(W^*)=-2\delta_i(p)$.
Thus \[
\delta_i(W^*)\delta_i(W)=4 \delta_i(p)\delta_i(p) \] and
\[
\delta^2_i(W^*)=\delta^2_i(W)=-2\delta^2_i(p).
\] 
Then
\[
\begin{split}
W^*(\Delta W)+\sum_{i=1}^2  \delta_i(W^*)\delta_i(W)=& 4 \sum_i \delta_i(p)\delta_i(p)-2(1-2p)(\Delta p)\\
=&2 [\sum_i \delta_i(p)\delta_i(p)-(1-2p)(\Delta p)]\\
=&2 \left[\Delta p -  p(\Delta p)-(\Delta p)p- (1-2p)(\Delta p)\right]\\
=& 2[p(\Delta p)-(\Delta p)p ]=0.
\end{split}
\] 
\end{proof}

\begin{theorem}\label{T: estimate}
Let $W$ be the above harmonic unitary. Then the energy $E(W)\ge 4\pi$. Moreover, we can actually obtain the equality.  
\end{theorem}
\begin{proof}
Recall that the action functional $S_D(\phi)$ corresponding to two points target space is given by $\Tr(\delta_1(p)\delta_1(p)+\delta_2(p)\delta_2(p))$ where the field map $\phi$ is determined by the projection $p$ in $\nc$.\\
 Since the action functional $\mathcal{L}_D(\psi)$ for the field map $\psi: C(S^1) \to \nc$ is $\Tr(\delta_1(U^*)\delta_1(U)+\delta_2(U^*)\delta_2(U))$ where $\psi(z)=U$,  in case of $\psi(z)=U=1-2p$  we have
\[\mathcal{L}_D(\psi)=4S_D(\phi).\]
Moreover, 
\[S_D(\phi)\ge -2\pi c_1(p).\]
By the integrality property of the first Chern number, 
\[\mathcal{L}_D(\psi)\ge 8\pi.\]
Let us show that we can choose $p$ such that $\tau(p)=|\theta|$, then  it follows that $c_1(p)=-1$.\\
 Recall that the instanton solution is given by $p=_A<\xi<\xi,\xi>_B^{-1/2}, \xi<\xi,\xi>_B^{-1/2}>$ for the gaussians $\xi \in \Xi$. Then 

\[\begin{split}
\tau_B(_B<\xi<\xi,\xi>_B^{-1/2}, \xi<\xi,\xi>_B^{-1/2}>)&=|-\frac{1}{\theta}| \tau_A(_A<\xi<\xi,\xi>_B^{-1/2}, \xi<\xi,\xi>_B^{-1/2}>)\\
&=|-\frac{1}{\theta}| \tau(p)\\
&=1
\end{split}\]
So $\tau(p)=|\theta|$ and we are done.
\end{proof} 
Recall that Li \cite{Li} showed that the scalar multiples of $U^mV^n$ are minima in the connected component on the projective unitary group, which is the unitary group modulo $\mathbb{T}$. But in view of the above theorem these are not global minima.  
\begin{corollary}
The scalar multiples of $U^mV^n$ are not global minima for the energy functional $\mathcal{L}_D$ or $2E$.
\end{corollary}
\begin{proof}
Note that if $E(U^mV^n)=2\pi^2(m^2+n^2)$. Thus $\mathcal{L}_D(\phi)=4\pi^2(m^2+n^2)\ge 4\pi^2 > 8\pi$,  where $\phi$ sends $z$ to the scalar multiples of $U^mV^n \in \nc$.
\end{proof}

\section{A group action on field maps}
\begin{lemma}\label{L:action}
Suppose $w$ is an element  in $\nc$.
Either $w^*\delta_j(w)$ or $w\delta_j(w^*)$ is scalar valued if and only of $w$ is the scalar multiple of $U^mV^n$ for some $m,n \in \mathbb{Z}$.  
\end{lemma}
\begin{proof}
One direction is obvious. Let us consider the Fourier expansion of $w$,
 i.e., $w=a_{m\,n}U^mV^n+\sum_{(k,l)\neq(m,n)}a_{k\,l}U^kV^l$ and assume at least one $a_{k\,l}$ is non-zero in the sum. Then 
\[\begin{split}
&w^*\delta_1(w)=(a_{m\,n}U^mV^n+\sum_{(k,l)\neq(m,n)}a_{k\,l}U^kV^l)^*(2\pi im a_{m\,n}U^mV^n+\sum_{(k,l)\neq(m,n)} 2\pi i k a_{k\, l}U^kV^l)\\
&=(\overline{a_{m\,n}}(V^*)n(U^*)^m+\overline{a_{k,l}}(V^*)^l(U^*)^k+\cdots)(2\pi im a_{m\,n}U^mV^n+2\pi i k a_{k\, l}U^kV^l+\cdots)\\
&=2\pi im |a_{m\,n}|^2+2\pi im a_{m\,n}e^{2\pi i \theta(m-k)l}\overline{a_{k\,l}}U^{m-k}V^{n-l}+2\pi ik e^{2\pi i \theta (k-m)n}\overline{a_{m\,n}}a_{k\,l}U^{k=m}V^{l-n}+\cdots
\end{split}\]
Thus  $w^*\delta_1(w)$ cannot be scalar if $w$ is not the single term $U^mV^n$ up to the scalar multiple. 
The other statements about $w\delta_j(w^*)$ is proved similarly. 
\end{proof}
\begin{remark}
Note that when either $w^*\delta_i(w)$ or $w\delta_i(w^*)$ becomes scalar, these values are pure imaginary numbers so that $w^*\delta_i(w)+\delta_i(w^*)w=0$ and $w\delta_i(w^*)+\delta_i(w)w^*=0$. This observation will be used later.    
\end{remark}
For any pair $(m,n)\in \mathbb{Z}^2$, let $w$ be the scalar multiple of $U^mV^n$. 
Recall that $\Ad w$ is an inner automorphism sends $a \in \nc$ to $waw^*$.  Then we can introduce the action of $\mathbb{Z}^2$ on field maps to $\nc$ via the composition with the inner automorphism $\Ad w$ . In the following we show that the action generates a family of critical points starting at a critical point.
\begin{theorem}
Let $p$ be an instanton solution of $p(\Delta p)-(\Delta p)p=0$. Then $\Ad w (p)$ is also an instanton solution.
\end{theorem} 
\begin{proof}
Without loss of generality we can assume that $w=U^mV^n$. If we let $p$ be the instanton solution and  $q=wpw^*$, then
\[\delta_i(q)=\delta_i(w)pw^*+w\delta_i(p)w^*+wp\delta_i(w^*).\]
Therefore, we have 
\[\begin{split}
\delta_i^2(q)&=\delta^2_i(w)pw^*+\delta_i(w)\delta_i(p)w^*+\delta_i(w)p\delta(w^*)\\
&+\delta_i(w)\delta_i(p)+w\delta^2_i(p)w^*+w\delta_i(p)\delta_i(w^*) \\
&+\delta_i(w)p\delta_i(w^*)+w\delta_i(p)\delta_i(w^*)+wp\delta^2_i(w^*).
\end{split}\]
It follows that 
\begin{align*}\label{E:instanton}
&q(\Delta q)-(\Delta q)q=\sum_i [wpw^*\delta^2_i(w)pw^*+2wpw^*\delta_i(w)\delta_i(p)w^*+2wpw^*\delta_i(w)p\delta(w^*)\\
&+wp\delta_i^2(p)w^*+2wp\delta_i(p)\delta_i(w^*)+wp\delta^2_i(w^*) ]-\sum_i [\delta^2_i(w)pw^* +2\delta_i(w)\delta_i(p)pw^* \\ 
&+2\delta_i(w)p\delta_i(w^*)wpw^*+w\delta^2_i(p)pw^*+2w\delta_i(p)\delta_i(w^*)wpw^*+wp\delta^2_i(w^*)wpw^*]\\
\intertext{Since the sencond and fifth terms are cancelled out in each bracket and the remaining corrsponding terms from each bracket are cancelled out, it can be simplified as follows;} 
&=w(p(\Delta p)-(\Delta p)p)w^*=0.  
\end{align*}
\end{proof}  
In fact, it was observed in \cite{DKL} that the Gaussian solutions $\xi_{\lambda}$, $\xi_{\lambda'}$ of the self-duality equation are gauge unvariant under the right action of $A_{-1/ \theta}$ on the module if and only if $\xi_{\lambda'}=\xi_{\lambda}U_1^mV_1^n$ for some integers $m,n$ where $U_1$ and $V_1$ are generators of $A_{-1/ \theta}$. But we do not know if the fact that $wpw^*$ is the instanton solution impies that $w$ is of the form $U^mV^n$.    
\begin{theorem}\label{T:harmonicunitary}
Suppose that $v$ is a harmonic unitary. Then $\Ad w (v)$ is also a harmonic unitary. 
\end{theorem}
\begin{proof}
Let $u$ be the generator of $C(S_1)$ and let $v$ be the image of $u$  for the field map $\phi:C(S_1) \to \nc$. Let $\psi= \Ad w \circ \phi $. Note that 
\[(\psi(u))^*\delta_j(\psi(u))=w^*\delta_j(w)+wv^*\delta_j(v)w^*+w\delta_j(w^*).\] 
By the assumption $w$,  it follows that 
\[\delta_j[ (\psi(u))^*\delta_j(\psi(u))]=\delta_j(w)v^*\delta_j(v)w^*+w\delta_j(v^*\delta_j(v))w^*+wv^*\delta_j(v)\delta_j(w^*).\]
If $\delta_j(w)C=D\delta_j(w)$ for some $C, D \in \nc$, then $wC=Dw$, and vice versa. Thus the first term and the third term in the above  are cancelled out since $\delta_j(w)w^*+w\delta_j(w^*)=0$.  
Consequently, 
\[\sum_j \delta_j(\psi(u^*) \delta_j(\psi(u)))=w\left[\sum_j \delta_j(v^*\delta_j(v))\right]w^*=0.\]
\end{proof}
Since $E(u)$ is invariant under the multiplication of $u$ by a scalar $\lambda \in \mathbb{T}$, there is a natural gauge action of $\mathbb{T}$ on harmonic unitaries. Denote by $[u]$ the class in the projective unitary group $PU(\nc)=U(\nc)/\mathbb{T}$.  
\begin{corollary}
$[u]$ is invariant under the $\mathbb{Z}^2$-action via the inner automorphism $\Ad w$. In other words, $\Ad w (u)$ is the scalar multiple of $u$. 
\end{corollary}
Similarly we can introduce $\mathbb{Z}^2$- actions on field maps on $\nc$ via $w$.  
\begin{theorem}
Suppose $\phi:\nc \to \nc$ is a harmonic map. Then $\Ad w \circ \phi$ is also a harmonic map.
\end{theorem}
\begin{proof}
Let $\phi$ be a harmonic map and $\psi=\Ad w \circ \phi$. To prove the claim, we must show taht for every $A', B'$ satisfying $A'-\psi(V^*)A'\psi(V)=B'-\psi(U^*)B'\psi(U)$  the following equation holds.
\begin{equation}\label{E:harmonic1}
\sum_j \Tr(A'\delta_j\left[ \psi(U^*)\delta_j(\psi(U))\right])+\sum_j \Tr(B'\delta_j \left[ \psi(V)\delta_j(\psi(V))\right])=0
\end{equation}
First, note that If we take $A=w^*A'w$, $B=w^*B'w$, then $A-\phi(V^*)A\phi(V)=B-\phi(U^*)B\phi(U)$. Since $\phi$ is a harmonic map, 
\[\sum_j \Tr(A\delta_j\left[ \phi(U^*)\delta_j(\phi(U))\right])+\sum_j \Tr(B\delta_j \left[ \phi(V)\delta_j(\phi(V))\right])=0\] 
holds. Then using invariance of the trace under under cyclic permutation we have 

\begin{equation}\label{E:harmonic2}
\sum_j \Tr(A'w\delta_j\left[ \phi(U^*)\delta_j(\phi(U))\right]w^*)+\sum_j \Tr(B'w\delta_j \left[ \phi(V)\delta_j(\phi(V))\right]w^*)=0.
\end{equation}
Recall that $\delta_j[\psi(U^*)\delta_j(\psi(U))]=w\delta_j[\phi(U^*)\delta_j(\phi(U))]w^*$ as in the proof of Theorem \ref{T:harmonicunitary}. Simliarly, $\delta_j[\psi(V^*)\delta_j(\psi(V))]=w\delta_j[\phi(V^*)\delta_j(\phi(V))]w^*$. Therefore (\ref{E:harmonic2}) is equal to (\ref{E:harmonic1}).
\end{proof}
There is a  gauge action of $\mathbb{T}^2$ on harmonic maps $\phi$ by sending $U$ to $\lambda \phi(U)$ and $V$ to $\mu \phi(V)$ for $(\lambda, \mu) \in \mathbb{T}^2$ . Denote by $[\phi]$ the orbit of the harmonic map $\phi$ under the gauge action.  

\begin{corollary}
Let $\phi_A$ be a harmonic map on $\nc$ associated with an element $A \in \SL(2, \mathbb{Z})$. Then  $[\phi_A]$ is  invariant under the action of $\mathbb{Z}^2$.
\end{corollary}
Finally we check  the $\mathbb{Z}^2$-action on field maps from $\SU(2,\mathbb{C})$ to $\nc$. 
\begin{theorem}
Suppose $\phi:\SU(2,\mathbb{C}) \to \nc$ is a solution of the Euler-Lagrange equations in Proposition \ref{P:SU(2)}. Then $\Ad w \circ \phi$ is a solution of the Euler-Lagrange equations. 
\end{theorem} 
\begin{proof}

Let $\psi=\Ad w \circ \phi$ and  $A', B'$ such that 
\[   \psi(\a)A'\psi(\g)-\psi(\g)\psi(\a)A'=\psi(\g)B'\psi(\a)-\psi(\a)\psi(\g)B'\]
\[ \psi(\a)A'\psi(\g^*)-\psi(\g^*)\psi(\a)A'=\psi(\a)B'\phi(\g^*)-B'\psi(\g^*)\psi(\a)\]

If $A=w^*A'w$, and $B=w^*B'w$, then 
\[
\begin{split}
\phi(\a)A\phi(\g)-\phi(\g)\phi(\a)A&=\phi(\a)w^*A'w\phi(\g)-\phi(\g)\phi(\a)w^*A'w\\
&=w^*\left(\psi(\a)A'\psi(\g)-\psi(\g)\psi(\a)A' \right)w\\
&=w^*\left(  \psi(\g)B'\psi(\a)-\psi(\a)\psi(\g)B' \right)w\\
&=\phi(\g)B\phi(\a)-\phi(\a)\phi(\g)B
\end{split}
\]
Also, 
\[
\begin{split}
\phi(\a)A\phi(\g^*)-\phi(\g^*)\phi(\a)A&=\phi(\a)w^*A'w\phi(\g^*)-\phi(\g^*)\phi(\a)w^*A'w\\
&=w^* \left( \psi(\a)A'\psi(\g^*)-\psi(\g^*)\psi(\a)A' \right)w\\
&=w^*(\psi(\a)B'\psi(\g^*)-B'\psi(\g^*)\psi(\a))w\\
&=\phi(\a)B\phi(\a)-\phi(\a)\phi(\g)B
\end{split}
\]
Since $\phi$ is the solution of the Euler-Lagrange equation, for the above $A, B$ 
\[0=\sum_{j=1}^2 \{ \Tr(A\delta_j[\phi(\a)^*\delta_j(\phi(\a))])\}+\sum_{j=1}^2 \{ \Tr(B\delta_j[\phi(\g)^*\delta_j(\phi(\g))])\}.\]
On the other hand, 
\[\psi(\a^*)\delta_j(\psi(\a))= \text{scalar}+w\phi(\a^*)\delta_j(\phi(\a))w^*.\]
Thus, 
\[\begin{split}
\delta_j[ \psi(\a^*)\delta_j(\psi(\a))]& = \delta_j(w)[ \phi(\a^*)\delta_j(\phi(\a))]w^*+ w \delta_j[\phi(\a^*)\delta_j(\phi(\a))]w^*+ w[\phi(\a^*)\delta_j(\phi(\a))]\delta_j(w^*)\\
&=w \delta_j[\phi(\a^*)\delta_j(\phi(\a))]w^*
\end{split}
 \]
since the first term and the third term are cancelled out.

\[\begin{split}
&\sum_{j=1}^2 \{ \Tr(A' \delta_j[\psi(\a)^*\delta_j(\psi(\a))])\}+\sum_{j=1}^2 \{ \Tr(B' \delta_j[\psi(\g)^*\delta_j(\psi(\g))])\}\\
&=\sum_{j=1}^2 \{ \Tr(A'w\delta_j[\phi(\a)^*\delta_j(\phi(\a))]w^*)\}+\sum_{j=1}^2 \{ \Tr(B' w\delta_j[\phi(\g)^*\delta_j(\phi(\g))]w^*)\}\\
&=\sum_{j=1}^2 \{ \Tr(w^*A'w\delta_j[\phi(\a)^*\delta_j(\phi(\a))])\}+\sum_{j=1}^2 \{ \Tr(w^*B' w\delta_j[\phi(\g)^*\delta_j(\phi(\g))])\}\\ 
&=\sum_{j=1}^2 \{ \Tr(A\delta_j[\phi(\a)^*\delta_j(\phi(\a))])\}+\sum_{j=1}^2 \{ \Tr(B\delta_j[\phi(\g)^*\delta_j(\phi(\g))])\}=0.
\end{split}\]
Here we used the cylic property of $\Tr$ before the final step.

\end{proof}
Again there is a natrual guage action of $\mathbb{T}^2$ on coercive maps from $\SU(2,\mathbb{C})$ to $\nc$. Denote by $[\phi]$ the orbit of the coercive map $\phi$ under the guage action.  
\begin{corollary}
Let $\phi_A$ be a coercive map from $\SU(2,\mathbb{C})$ to $\nc$ associated with an element $A \notin \GL(2, \mathbb{Z})$. Then  $[\phi_A]$ is  invariant under the action of $\mathbb{Z}^2$.
\end{corollary}


\end{document}